\documentclass[12pt]{article}
\title{The Ginibre ensemble and Gaussian analytic functions}
\author{
Manjunath Krishnapur \ \and   B\'alint
Vir\'ag}
\date{\today}

\usepackage{palatino}
\usepackage{amsfonts}
\usepackage{graphicx}
\usepackage{bm}
\usepackage{amsmath}
\usepackage{amsthm}
\usepackage{color}

\usepackage[longnamesfirst]{natbib}
\theoremstyle{plain}
    \newtheorem{theorem}{Theorem}
    \newtheorem{fact}[theorem]{Fact}
    \newtheorem{lemma}[theorem]{Lemma}
    
    \newtheorem{corollary}[theorem]{Corollary}

\theoremstyle{definition} 
    
    \newtheorem{remark}[theorem]{Remark}

\theoremstyle{remark} 

\def\given{\left.\vphantom{\hbox{\Large (}}\right|}

\newcommand\mnote[1]{} 

\newcommand\be{\begin{equation}}

\newcommand\ee{\end{equation}}

\newcommand{\comment}[1]{}
\newcommand{\eps}{\varepsilon}

\newcommand{\CC}{{\mathbb C}}

\newcommand{\ev}{{\rm \bf  E}}

\newcommand{\one}{{\mathbf 1}}

\newcommand{\sm}{{\raise0.3ex\hbox{$\scriptstyle \setminus$}}}
\newcommand{\lcirc}{{\raise-0.15ex\hbox{$\scriptscriptstyle \circ$}}}

\newcommand{\g}{\mathcal N}

\newcommand{\RR}{{\mathbb R}}

\newcommand{\lstar}{{\raise-0.15ex\hbox{$\scriptstyle \ast$}}}

\renewcommand{\Re}{\operatorname{Re}}

\begin{document}
\maketitle \label{fig1}
\let\thefootnote\relax\footnotetext{M.K.  was supported in part by DST and UGC through DSA-SAP-PhaseIV.}

\begin{abstract}
We show that as $n$ changes, the characteristic polynomial of the
$n\times n$ random matrix with i.i.d.\ complex Gaussian entries
can be described recursively through a process analogous to
P\'olya's urn scheme. As a result, we get a random analytic
function in the limit, which is given by a mixture of Gaussian
analytic functions. This gives another reason why the zeros of
Gaussian analytic functions and the Ginibre ensemble exhibit
similar local repulsion, but different global behavior. Our
approach gives new explicit formulas for the limiting analytic
function.

\end{abstract}
\begin{center}
\includegraphics*[width=300pt]{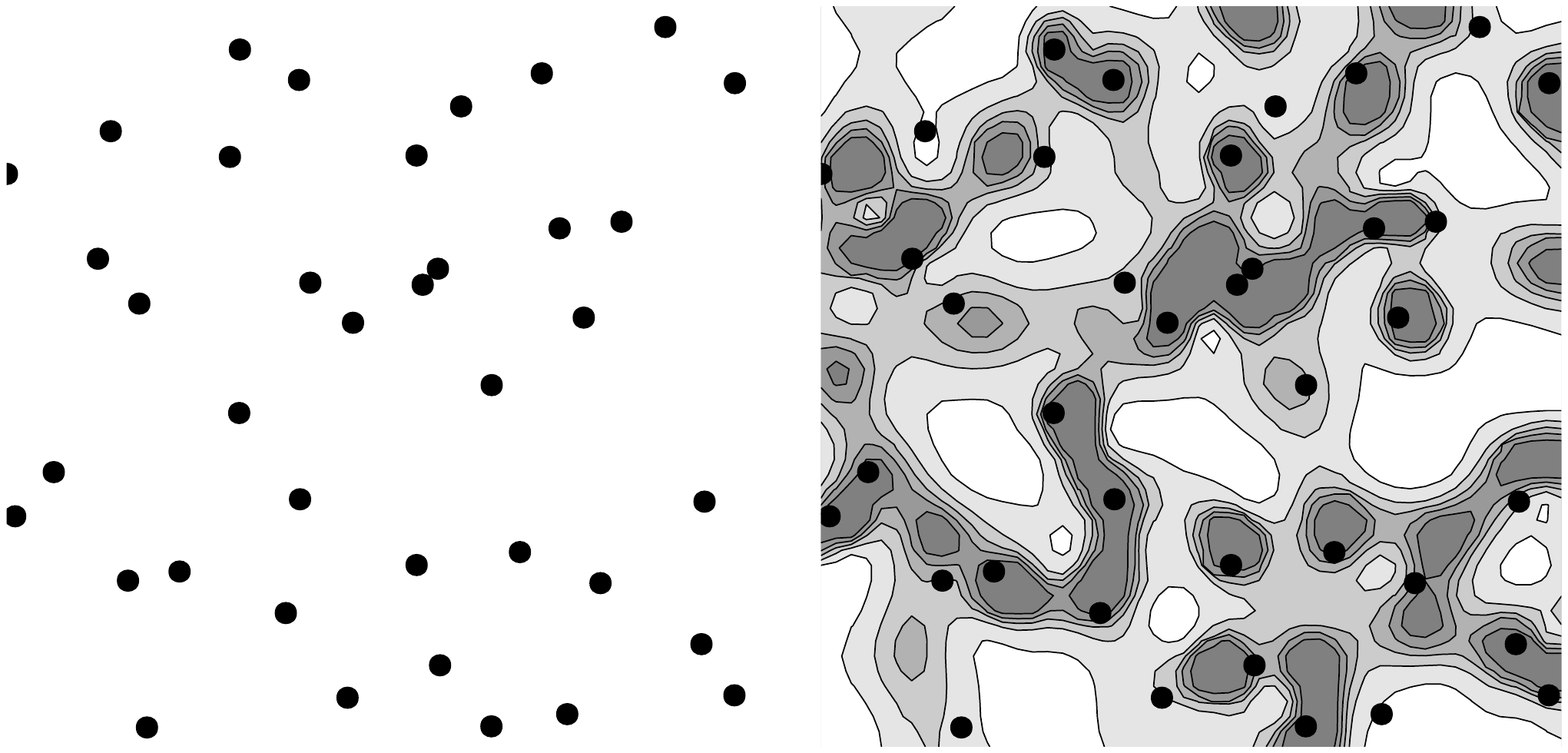}
\\
Ginibre points and their random intensity measure
\bigskip
\end{center}

\section{Introduction}

In studies of random point sets in the complex plane, two
canonical examples have emerged, sharing some features but
differing in others.

The first is the {\bf infinite Ginibre ensemble}, attained
as the limit of the {\bf finite Ginibre ensembles}, i.e.\
the set of eigenvalues of  $n\times n$ random matrices 
$A_n$ filled with independent standard complex Gaussian
entries. It can be thought of as the stationary distribution
of a system of particles performing planar Brownian motions but repelling each 
other through a drift given by their inverse distance.

The second is the zero set of the random power series $\sum
a_nz^n/\sqrt{n!}$ for independent standard complex
Gaussians $a_n$.  It is an example of a random analytic
function whose values are jointly centered complex
Gaussian, a {\bf Gaussian analytic function} or {\bf GAF},
for short. \cite{sodin} has shown that the intensity measure (under which the measure of a Borel subset of $\mathbb C$ is the expected number of zeros that fall in the set) of the zeros
 of such a function in fact determines the entire distribution of zeros. The
above power series, called {\bf planar GAF}, has intensity
measure given by a constant multiple of the Lebesgue measure. Sodin's theorem implies
 that this is the only Gaussian analytic function whose  zero set has a translation-invariant
distribution on $\mathbb C$.

Beyond translation invariance, these two processes share
some local properties. For example, in both cases we have
$$
P(\mbox{two points in a fixed disk of radius } \eps) \asymp
\eps^6
$$
where four of the six in the exponent come from the square of the area of the disk and the extra two comes from quadratic repulsion. Contrast this with Poisson process which only has $\eps^4$. This behavior is shared by all non-degenerate GAFs, see the work of \cite{NazarovSodin} on universality.

On the global scale, however, the two models are very
different. For smooth, compactly supported $\varphi:\CC\to
\RR$ with total integral zero we have two central limit 
theorems
$$ 
\begin{array}{lccc}
\mbox{for Ginibre:} &\sum_z
\varphi(z/n) &\Rightarrow&N(0,\frac{1}{4\pi}\|\nabla  \varphi\|^2),\\
\mbox{for planar GAF:} &
n \ \sum_z \varphi(z/n) &\Rightarrow&
N(0,c\|\Delta \varphi\|^2),
\end{array}
$$
where the sum is over all points of the processes. These results are due to 
\cite{RV} and \cite{ST1}, respectively.

The strikingly different central limit theorems show that the
global behavior of the two random point processes are very
different.

The goal of this paper is to prove a theorem which clarifies why this
phenomenon happens. 
We study the distributional limit of the
characteristic polynomial $Q_n(z)$ of $A_n$. It is known (see \cite{girko} or  \cite{kostlan})  that $|Q_n(0)|$, the absolute value of the determinant, has
the same distribution as the product of independent $\mbox{Gamma}(i,1)$ 
random variables with $i=1,\ldots, n$. Since $\log( \mbox{Gamma}(i,1))$ has variance asymptotic to $1/i$, we see that $\log |Q_n(0)|$, centered and  
divided  by $\sqrt{\log n}$ converges in distribution to a standard normal random variable.

A simple consequence of this fact is that there are no
constants $a_n, b_n$ so that the random variable
$(Q_n(0)-a_n)/b_n$ converges in law to a non-degenerate
limit. In contrast, we prove the following. Note that the convergence in distribution here is with respect to the topology of uniform convergence on compact subsets of the complex plane.

\begin{theorem}\label{thm:convergenceofQn}
There exist positive random variables $A_n$ so that  the 
normalized characteristic polynomial $ Q_n(z)/A_n$ converges in
distribution to a random analytic function $Q(z)$.

Moreover, there exists a random positive definite Hermitian
function $K(z, w)$ so that given $K$ the function $Q(z)$ is a
Gaussian field with covariance kernel $K$. Further, $K$ is analytic in $z$ and anti-analytic in $w$, hence $Q$ is a GAF conditional on $K$.
\end{theorem}

Thus, the limit is in fact a randomized Gaussian analytic 
function. Theorem~\ref{thm:convergenceofQn} thus gives a novel link between the
world of random matrices and Gaussian analytic functions.
In physics, certain connections between determinantal and Gaussian
structures are referred to as super-symmetry; Theorem~\ref{thm:convergenceofQn}
is a very specific, mathematically precise instance of a connection between these two structures.

A Gaussian analytic function (GAF) $f$ is a complex Gaussian field on $\CC$ 
which is almost surely analytic. By the theorem of \cite{sodin} referred to earlier,
the distribution of the zeros of a Gaussian analytic
function is determined by their intensity measure
\begin{equation}\label{ek}
d\mu(z)=\frac1\pi \Delta \log
K(z,z) d\mathcal L(z),
\end{equation}
where  $K(z,w)=\ev[f(z)\overline{f(w)}]$  is the covariance kernel and $\mathcal L$ is Lebesgue measure. Formula \eqref{ek}
is a special case of the well-known Kac-Rice formulas (see sec. 2.4 of \cite{hkpvbook}). In our setting this is just  an averaged 
version of Green's formula in complex analysis.  

A direct consequence is a connection between the infinite
Ginibre points and Gaussian analytic zero points, depicted
in the front-page figure.
\begin{corollary}
The infinite Ginibre point process has the same distribution as
Gaussian analytic zero points having a randomly chosen
intensity measure $\mu$.
\end{corollary}

Note that this does not make the infinite Ginibre point process
the zero set of a true Gaussian analytic function, only one with a
randomized covariance kernel. This randomization does not change
the qualitative local behavior, but changes the global one, which
explains the phenomenon discussed above.

We conclude our paper by computing the first moment (everywhere) and the second moment 
(at zero) of the limiting covariance kernel normalized so that these moments exist.  

\bigskip
{\bf Real Ginibre matrices.} With our methods, we also prove an analogous theorem for characteristic polynomials of real Ginibre matrices (having i.i.d. real standard normal entries). However there are a few modifications as the characteristic polynomials are complex-valued but the values are not complex Gaussian (as complex Gaussians are isotropic, by definition, for us).

\begin{theorem}\label{thm:convergenceofQnforrealginibre} Let $A_{n}$ be the real Ginibre matrix with i.i.d.~$N(0,1)$ entries.
There exist positive random variables $A_n$ so that  the 
normalized characteristic polynomial $ Q_n(z)/A_n$ converges in
distribution to a random analytic function $Q(z)$.

Moreover, there exist functions $K(z, w),\hat{K}(z,w)$ so that given $K,\hat{K}$,  the function $Q(\cdot)$ is a Gaussian field with $\ev[Q(z)Q^{*}(w)]=K(z,w)$ and $\ev[Q(z)Q(w)]=\hat{K}(z,w)$.  Further, $K$ is analytic in $z$ and anti-analytic in $w$ while $\hat{K}$ is analytic in both $z$ and $w$. Therefore, $Q$ is a random analytic function whose real and imaginary parts are jointly Gaussian fields.
\end{theorem}

For later purposes, we introduce the following notation.

\noindent{\bf Notation.} We write $Y\sim \mbox{Normal}[m(\lambda),L(\lambda,\mu),\hat{L}(\lambda,\mu)]$ to mean that the real and imaginary parts of $Y$ are jointly Gaussian fields, $\ev[Y(\lambda)]=m(\lambda)$ and $\ev[Y(\lambda)\bar{Y}(\mu)]=L(\lambda,\mu)$ and $\ev[Y(\lambda)Y(\mu)]=\hat{L}(\lambda,\mu)$.

Observe that $Y$ is a random analytic function if and only if $m(\lambda)$ is analytic in $\lambda$, $L(\lambda,\mu)$ is analytic in $\lambda$ and anti-analytic in $\mu$, and $\hat{L}$ is analytic in both $\lambda$ and $\mu$. Further,  $Y$ is a complex Gaussian field if and only if $\hat{L}=0$. In particular, for a GAF the third argument is identically zero. 

\bigskip 

{\bf A P\'olya's urn scheme for characteristic polynomials.} We
construct the random covariance kernel is via a version of
P\'olya's urn scheme. By similarity transformations, we first convert the matrix into a 
Hessenberg (lower triangular plus an extra off-diagonal) form.
Then, we consider characteristic polynomials of the successive
minors. It turns out that these polynomials develop in a version
of P\'olya's urn scheme, which we  recall here briefly. An urn 
containing a black and a white ball is given, and at each time an 
extra ball is added to the urn whose color is black or white with
probabilities proportional to number of the balls of the same
color already in the urn. In short, we may write

$$
X_1=0,\qquad X_2=1, \qquad X_{k+1}=\mbox{Bernoulli}
\left[\frac{X_1+\ldots +X_k}{k}\right], \quad k\ge 3
$$
For us, the essential part of P\'olya's urn is that the
 $X_{k+1}$ given the events up to time $k$ is a random variable
 whose mean is close to the average of the previous $X_k$. The
 conclusion is that this average converges almost surely to a
 random limit.

It turns out that a similar recursion this also holds for the
random constant multiples $X_n=X_n(\lambda)$ of the characteristic
polynomials $Q_n$. We shall see that 

$$ 
X_1 = 1,\qquad  X_{k+1}\given_{X_{1},\ldots,X_{k}} = \mbox{Normal}\left[m_{k}(\lambda),M_{k}(\lambda,\mu),\hat{M}_{k}(\lambda,\mu)\right]
$$
where 
\begin{align*}
m_{k}(\lambda)&=\frac{\lambda X_k(\lambda)}{b_k}, \qquad M_{k}(\lambda,\mu)=\frac{X_1(\lambda)X_1^*(\mu)+\ldots +X_k(\lambda)X_k^*(\mu)}{b_k^2} \\
 \hat{M}_{k}(\lambda,\mu)&=\begin{cases}
 b_{k}^{-2}\left(X_1(\lambda)X_1(\mu)+\ldots +X_k(\lambda)X_k(\mu)\right) & \mbox{ if }\beta=1.\\
 0 &\mbox{ if }\beta=2.
 \end{cases}
\end{align*}
with $b_k^2/k \rightarrow 1$.

The main feature of P\'olya's
urn is that the parameter of the Bernoulli distribution
converges to a random limit almost surely, and the samples
are asymptotically independent from this random
distribution. Similarly, in our case, the variance
parameter converges almost surely, and given the limit, the
samples are asymptotically independent. In particular, as 
we shall show, in the limit they behave like a Gaussian 
analytic function with a random covariance kernel.

\section{The recursion for the characteristic polynomial}\label{sec:recursionofcharpoly}

Start with the real ($\beta=1$) or complex ($\beta=2$)  Ginibre matrix, having i.i.d real or complex Gaussian entries.  For the purposes of this paper, a standard complex Gaussian random variable will mean $X+iY$ where $X,Y$ are i.i.d. $N(0,1/2)$. 

Following the randomized application of the Lanczos algorithm as
pioneered by \cite{trotter}, we conjugate the matrix $A$ by an
orthogonal/unitary block matrix as follows. Let
$$A=\left(%
\begin{array}{cccc}
  a_{11} & \ldots  & b & \ldots \\
  \vdots & &  &  \\
  c & & A_{11} &  \\
  \vdots &  & & \\
\end{array}\right),\qquad
O=\left(%
\begin{array}{cccc} 1 & \ldots  & 0 & \ldots \\
  \vdots & &  &  \\
  0 & & O_{11} &  \\
  \vdots &  & & \\
\end{array}\right)
$$
so that we get
$$OAO^*=\left(%
\begin{array}{cccc} a_{11} & \ldots  & b O_{11}^*  & \ldots \\
  \vdots & & &  \\
  O_{11} c & & O_{11}A_{11}O_{11}^* &  \\
  \vdots &  & & \\
\end{array}\right).
$$
If $O_{11}$ is chosen depending only on $b$ so that it
rotates $b$ into the first coordinate vector, then we get a matrix
of the form
$$
OAO^*=\left(%
\begin{array}{ccccc} \mathcal N  &   \frac{1}{\sqrt{\beta}}\chi_{(n-1)\beta}& 0 &\cdots &0\\
  \mathcal N  & \mathcal N &    \cdots &&\mathcal N\\
  \vdots  & &  &  &\vdots\\
  \mathcal N  & \mathcal N  &  \cdots && \mathcal N
\end{array}\right)
$$
where all the entries are independent, and $\mathcal N$ indicates that the entry has Normal (real or complex) distribution and  $\chi_n$ indicates that the entry has  the distribution of the
length of the vector with independent standard real
Gaussian entries in $n$ dimensions. This is because the
normal vector $c$ has a distribution
that is invariant under rotation and the matrix $A_{11}$ has a distribution
that is invariant under conjugation by a 
rotation. Repeated application of this
procedure (where the rotation matrices are block diagonal
with an identity matrix of increasing dimension at the top)
brings $M$ to a form
$$
\left(%
\begin{array}{ccccc}
  \g     &   \frac{1}{\sqrt{\beta}}\chi_{(n-1)\beta} &  &  & 0 \\
  \g      & \g  &  \frac{1}{\sqrt{\beta}}\chi_{(n-2)\beta}&  &  \\
  \vdots &   & & \ddots  &  \\
         &   &  &  & \frac{1}{\sqrt{\beta}}\chi_\beta \\
  \g     &   & \ldots &  & \g \\
\end{array}%
\right)
$$
We conjugate this matrix by the reverse permutation matrix
and transpose it. The eigenvalue equation of the resulting matrix reads
\begin{equation}\label{eq:hessenbergformofginibre}
\left(%
\begin{array}{ccccc}
  \g     &   \frac{1}{\sqrt{\beta}}\chi_{\beta} &  &  & 0 \\
  \g      & \g  &  \frac{1}{\sqrt{\beta}}\chi_{2\beta}&  &  \\
  \vdots &   &  & \ddots &  \\
         &   &  &  & \frac{1}{\sqrt{\beta}}\chi_{(n-1)\beta} \\
  \g     &   & \ldots &  & \g \\
\end{array}%
\right)X=\lambda X.
\end{equation}
\begin{remark} This reduction is analogous to the tridiagonal form of the GUE matrix obtained by \cite{trotter} and that has been of much use in studying the scaling limits of eigenvalues in recent years (for example, see \cite{edelman}, \cite{ramirezridervirag} and \cite{valkovirag}). 

\end{remark}

The matrix on the left of \eqref{eq:hessenbergformofginibre} is the one whose eigenvalues we will study. Indeed, let
$\g_k$ denote the $k$th row of this matrix, with the
$\chi$ variable removed. Then the $k$th row of the
eigenvalue equation is
\begin{equation}\label{eq:evalueequation}
\overrightarrow{X_k} \cdot  \g_k+\frac{1}{\sqrt{\beta}}\chi_{\beta k}X_{k+1}  =  \lambda X_k
\end{equation}
where $\overrightarrow{X_k}= (X_1,\ldots,
X_{k})$. For any $\lambda\in \mathbb C$, let
$X_1(\lambda)=1$, and define $X_{k+1}(\lambda)$ recursively as
the solution to the equation given above for $k< n$. Then, $\lambda$ is an eigenvalue of the matrix if and only if it satisfies the last equation $\overrightarrow{X_n}(\lambda) \cdot  \g_n=\lambda X_{n}(\lambda)$ or equivalently, if we solve for $k=n$ also and get $X_{n+1}(\lambda)=0$. 


But these equations are consistent as $n$ changes, so we may define the
$X_k(\lambda)$ for all $k\ge 0$ through the infinite version of the above matrix. Clearly $X_k$ is a polynomial of degree $k-1$ and hence,   $X_{k+1}(\lambda)$ is a random constant times the characteristic polynomial of the top $k\times k$ submatrix for every $k$.

We introduce the random functions
\begin{align*}
M_k(\lambda,\mu) &= \overrightarrow{X_k}(\lambda)\cdot
\overrightarrow{X_k^*}(\mu)=\sum_{j=1}^{k}X_{j}(\lambda)X_{j}^{*}(\mu), \\
\hat{M}_k(\lambda,\mu) &= \begin{cases}
\overrightarrow{X_k}(\lambda)\cdot
\overrightarrow{X_k}(\mu)=\sum_{j=1}^{k}X_{j}(\lambda)X_{j}(\mu) & \mbox{ if }\beta=1, \\
0 & \mbox{ if } \beta=2. 
\end{cases}
\end{align*}
where we use both $z^*$ and $\bar{z}$ to denote the complex conjugate of $z$. 

Let $\mathcal F_k$ denote the $\sigma$-field generated by
the first $k-1$ rows of the matrix together with
$\chi_{k\beta}$. From \eqref{eq:evalueequation} note that  $X_{k+1}$ given $\mathcal F_k$
is a Gaussian field with mean $\lambda \sqrt{\beta}X_k/\chi_{k\beta}$ covariance structure given by 
$\beta M_k/\chi_{k\beta}^2$ and $\beta \hat{M}_k/\chi_{k\beta}^2$.
Moreover, we have
$$
M_{k+1}-M_k = X_{k+1}(\lambda)\overline{
X_{k+1}(\mu)}, \;\; \hat{M}_{k+1}-\hat{M}_k = \begin{cases}X_{k+1}(\lambda)
X_{k+1}(\mu) & \mbox{ if }\beta=1, \\
0 & \mbox{ if } \beta=2. 
\end{cases}.
$$
Thus the evolution of $X_k$ can be summarized as a
randomized recursion with $X_{1}=1$ and 
$$
  X_{k+1}\given_{\mathcal F_{k}} = \begin{cases}
\mbox{Normal}\left[\frac{\lambda
X_k}{\frac{1}{\sqrt{\beta}}\chi_{k\beta}}, \frac{X_1X_1^*+\ldots
+X_kX_k^*}{\frac{1}{\beta}\chi_{k\beta}^2}\right]  & \mbox{ if } \beta=2, \\
\mbox{Normal}\left[\frac{\lambda
X_k}{\frac{1}{\sqrt{\beta}}\chi_{k\beta}}, \frac{X_1(\lambda)X_1^*(\mu)+\ldots
+X_k(\lambda)X_k^*(\mu)}{\frac{1}{\beta}\chi_{k\beta}^2},\frac{X_1(\lambda)X_1(\mu)+\ldots
+X_k(\lambda)X_k(\mu)}{\frac{1}{\beta}\chi_{k\beta}^2}\right]  & \mbox{ if } \beta=1.
\end{cases}
$$ 
for $k\ge 1$.

This is the recursion, analogous to P\'olya's urn discussed
in the introduction. Next, we establish a framework for the
asymptotic analysis of such recursions.

\section{P\'olya's urn in Hilbert space -- theorem and examples}

The main tool for the analysis of P\'olya's urn schemes
will be the following theorem.

\begin{theorem}[P\'olya's urn in Hilbert space]\label{thm Hilbert
Polya} Let $\mathcal H$ be a Hilbert space, and let $R_1,R_2,\ldots R_{k_0-1}$ be
deterministic elements of $\mathcal H$. Let $\mathcal F_k$ be a filtration in some probability space and for each $k\ge k_0$ assume that the $\mathcal H$-valued  random variables $R_k, M_k\in \mathcal F_k$ satisfy
\begin{eqnarray} \nonumber
M_k &=& \frac{R_1+\ldots+ R_k}{k}\\
\|\ev [R_{k+1}| \mathcal F_k] -M_k \| &=&
(1+\|M_k\|)O(\eps_k) +\|R_k\|O(1/k) \label{resample}
\\
 \ev [\|R_{k+1}\|^2| \mathcal F_k] &=&O(1+\|M_{k-1}\|^2+\|M_{k}\|^2)
 \label{second}
\end{eqnarray} 
for some positive sequence  $\eps_k$ such that $\eps_k
k^{-1}$ is summable.

Then $M_k$ converges to some limit a.s.\ and in $L^2$.
\end{theorem}

The right hand sides of \eqref{resample} and \eqref{second}
should simply read as  ``small''. The specific error terms
here are taylored to the application at hand.

Note that without the error term, \eqref{resample} says that
the mean of $R_{k+1}$ given $\mathcal F_k$ is equal to
$M_k=(R_1+\ldots + R_k)/k$. This is the setting of
P\'olya's urn scheme.

\subsection{Example: Classical P\'olya's urn}

Let $\mathcal H=\mathbb R$ and  $R_1=0$, $R_2=1$.  Given $R_1,\ldots, R_k$ let
$$
R_{k+1}\sim \mbox{Bernoulli}(M_k), \qquad \mbox{ with }
\qquad M_k=\frac{R_1+\ldots+ R_k}k.
$$ 
Then \eqref{resample} holds without error terms and
$R_{k+1}^2$ is bounded by 1. So $M_k$ converges to a limit $M$
almost surely.

It is well-known that $M$ has a Beta distribution, and the limiting distribution of $R_k$ is that of Bernoulli($M$) (i.e., sample $M$ and sample from Bernoulli $M$). In fact, for each $k$, the distribution of $R_k$ given $M$ is Bernoulli($M$) but in our setting this is a special phenomenon owing to the fact that error terms vanish.

\subsection{Example: a semi-random Hessenberg matrix}\label{ex:semirandom}

For the next example, let $D\subset \mathbb C$ be a closed disk. For our next example, we will consider the Hilbert space of $2$-variable functions from $D^2\to \mathbb C$, with
the following inner product
\begin{equation}\label{e:kernelnorm}
\langle f,g \rangle = \int_{D^2} f(x,y)\bar g(x,y)\,dx\,dy + \int_D f(x,x)\bar g(x,x)\,dx
\end{equation}
or, in other words, consider the usual $L^2$ inner product with respect to the sum of the Lebesgue measure on $D^2$ and Lebesgue measure on the diagonal of $D^2$. More precisely, we define this inner product on smooth functions and then take the completion to get a Hilbert space. We will denote the corresponding norm simply by $\|\cdot \|$. The following simple Lemma is needed. 

\begin{lemma}\label{l:knorm} Let $X\sim \mbox{Normal}[0,M,\hat{M}]$ and let $m: D\to \mathbb C$ be a function. Let $R=(X+m)\otimes (\bar X+\bar m)$ and $\hat{R}=(X+m)\otimes (X+m)$. Then
$$
\max\{\ev \|R\|^2 , \ev \|\hat{R}\|^2\}\le 8(\|m\|^{4}+\|m^2\|^2) + 24 (1+|D|)\left(\|M\|^2+\|\hat{M}\|^{2}\right).
$$
Here $|D|$ is the area of $D$ and the norms on $m,m^2$ are $L^2(D)$.
\end{lemma}


\begin{proof}
For $h:D\to \mathbb C$,  the norm \eqref{e:kernelnorm} gives 
\begin{equation}\label{e:rank 1 H-norm}
\|h \otimes h^{*}\|^{2}=\|h\otimes h\|^2= \|h\|^4 + \|h^2\|^2
\end{equation}
with the norms on the right hand side are for the usual $L^2(D)$. 
Using this, and triangle inequality we get 
\begin{eqnarray*}
\|R\|^2 &=& \|X+m\|^4 + \|(X+m)^2\|^2\\
&\le & 8\|X\|^4 + 8\|m\|^4+ (\|X^2\|+\|2Xm\|+\|m^2\|)^2
\end{eqnarray*}
Since $\|2Xm\|^2 \le \|X^2\|+ \|m^2\|$, we get the bound
$$
\|R\|^2 \le 8(\|X\|^4+\|X^2\|^2) + 8(\|m\|^4+\|m^2\|^2)
$$
By Cauchy-Schwarz
$$
\|X\|^2 \le \|\one\|\|X^2\|=\sqrt{|D|}\;\|X^2\|, 
$$
and 
$$\ev \|X^2\|^2 = \int_D \ev |X(z)|^4\,dz = \eta_{4,\beta}\left( \int_D  M^2(z,z)\,dz + \int_D  |\hat{M}|^2(z,z)\,dz \right)\le 3\left(\|M\|^2+\|\hat{M}\|^2\right)
$$
where $\eta_{4,\beta}=\ev |N^4|$ for a standard $\beta$-Gaussian $N$ (i.e.\ 3,2 for $\beta=1,2$ respectively).
Putting all these together, we get
\[\ev \|R\|^2 \le 8(\|m\|^4+\|m^2\|^2) +24 (1+|D|)\left(\|M\|^2+\|\hat{M}\|^2\right).
\qedhere \]
\end{proof}

Let $b_k$ be deterministic positive numbers such that   $\epsilon_k:=\frac{|b_k^2-k|}{k}$ satisfy $\sum k^{-1}\epsilon_k<\infty$. One example is to take $b_k=\sqrt{k}$.

Consider the nested matrices
$$
\left(%
\begin{array}{ccccc}
  \g     &   b_1&  &  & 0 \\
  \g      & \g  &  b_2&  &  \\
  \vdots &   &  & \ddots &  \\
         &   &  &  & b_{n-1} \\
  \g     &   & \ldots &  & \g \\
\end{array}%
\right)
$$
where the $\mathcal N$ refer to different  i.i.d.\ standard real
$(\beta=1)$ or complex $(\beta=2)$ normal random variables.

Just as in section~\ref{sec:recursionofcharpoly},  for $\lambda\in \mathbb C$, let
$X_1(\lambda)=1$, and define $X_{k+1}(\lambda)$ recursively as the
solution of the eigenvalue equation given by the $k$th row of
the matrix above (for $k< n$). In other words, let $\overrightarrow{\g_k}$ denote
the vector formed by the first $k$ entries of the $k$th row of the
nested matrices, write $\overrightarrow{X_k}(\lambda)= (X_1(\lambda),\ldots,
X_k(\lambda))$, and recursively define $X_{k+1}(\lambda)$ as the solution to the $k$th row of the eigenvalue equation
$$
\overrightarrow{\g_k}\cdot \overrightarrow{X_k}(\lambda) +
b_kX_{k+1}(\lambda) = \lambda X_k(\lambda).
$$
Let
$\mathcal F_k$ be the sigma-field generated by the first $k-1$
rows of the matrix. Define 
$$R_k(\lambda,\mu)= (X_k\otimes X_k^*)(\lambda,\mu)=X_{k}(\lambda)X_{k}^{*}(\mu), \;\; \hat{R}_k(\lambda,\mu)= (X_k\otimes X_k)(\lambda,\mu)=X_{k}(\lambda)X_{k}(\mu).$$

Then,  given $\mathcal F_k$ we have 
 \begin{eqnarray}
 X_{k+1} &\sim& \mbox{Normal} \left[ \frac{\lambda}{b_k}X_k, \frac{k}{b_k^2}M_k, \frac{k}{b_k^2}\hat{M}_k \right], \qquad \mbox{ where }  \label{e:meancov} \\
 M_k=\frac{R_1+\ldots+ R_k}k, && \hat{M}_k=\begin{cases}(\hat{R}_1+\ldots+ \hat{R}_k)/k & \mbox{ if }\beta=1,\\ 0 & \mbox{ if }\beta=2. \end{cases} \nonumber
 \end{eqnarray}
This means that conditionally on $\mathcal F_k$ the random
variable $X_{k+1}$ is a  Gaussian field  with the given mean function and 
covariance structure. For $\beta=2$  it is a Gaussian analytic function, while for $\beta=1$ it is a random anaytic function with jointly Gaussian real and imaginary parts.

In order to set  up a Hilbert space $\mathcal H$, we first fix a closed disk
$D\subset \mathbb C$, and consider the norm \eqref{e:kernelnorm}. Let $\mathcal H_{\beta}=\mathcal H$ for $\beta=2$ and $\mathcal H_{\beta}=\mathcal H\times \mathcal H$ for $\beta=1$ (the inner product on $\mathcal H\times \mathcal H$ is of course $\langle(u,v),(u',v')\rangle=\langle u,u'\rangle \langle v,v'\rangle$).

First consider the case $\beta=2$ where we can forget $\hat{R}_{k},\hat{M}_{k}$. Regard $R_k,M_k$ as random variables taking values in $\mathcal H$, and $X_k$ as an  $L^2(D)$-valued random variable.  Then
\[ \ev [R_{k+1}(\lambda,\mu) \, | \, \mathcal{F}_k] = \frac{k}{b_k^2}M_k(\lambda,\mu)
+ \frac{\lambda\overline \mu}{b_k^2} X_k(\lambda)\overline{X_k(\mu)} =
\frac{k}{b_k^2}M_k(\lambda,\mu)  + \frac{\lambda\overline \mu}{b_k^2} R_k(\lambda,\mu).
\]
Thus, condition \eqref{resample} follows because
 \begin{eqnarray*} \|\ev [R_{k+1} \, | \, \mathcal{F}_k]-M_k\|
 &\le&
\frac{|b_k^2-k|}{b_k^2}\|M_k\| + \|R_k\|O(1/k)
 \\&=&
O(\eps_k)\|M_k\| + \|R_k\|O(1/k)
 \end{eqnarray*}
To check condition \eqref{second}, we condition on $\mathcal F_k$ and use Lemma \ref{l:knorm}. Note the conditional mean and variance \eqref{e:meancov} of the Gaussian field $X_{k+1}$. We get
\begin{eqnarray*}
\ev [ \|R_{k+1}\|^2 \, | \, \mathcal F_k] &\le &
\frac{24k^2}{b_k^4}(1+|D|)\|M_k\|^2 + \frac{8}{b_k^{4}}(\|(\lambda X_{k})^2\|^2+\|(\lambda X_{k})^4\|)
\\
&\le &
c\|M_k\|^2 + \frac{c}{k^2}(\|X_{k}^2\|^2+\|X_{k}^4\|)
\\
&=&
c\|M_k\|^2 + \frac{c}{k^2}\|R_{k}\|^2,
\end{eqnarray*}
where $c$ depends on the sequence $b_k$ and  $D$ only, and the last equality follows from \eqref{e:rank 1 H-norm}. Writing $R_k$ as $kM_k-(k-1)M_{k-1}$ we get the required upper bound $c'(\|M_k\|^2 + \|M_{k-1}\|^2)$.

Theorem \ref{thm Hilbert
Polya} implies that $M_k$ converges (and uniformly on $D$) almost surely to a limit $M$.
Since local $L^2$ convergence for analytic functions implies
sup-norm convergence, the limit $M$ is analytic in its two
variables. Also
$$
\frac{R_k}{k} = \frac{k M_k -(k-1)M_{k-1}}{k} \to 0
$$
and so $X_k/b_k\to 0$.  The conditional law of $X_{k+1}$
given $\mathcal F_k$ converges to $\mbox{Normal}[0,M,0]$. As each $M_k(\lambda,\mu)$ is   analytic in the first variable and anti-analytic in the second, the same holds for $M$ and it follows that  $\mbox{Normal}[0,M,0]$ is in fact a  Gaussian analytic function, conditional upon $M$.

For $\beta=1$, we consider the $(R_{k},\hat{R}_{k}),(M_{k},\hat{M}_{k})$ as elements of $\mathcal H_{1}=\mathcal H\times \mathcal H$. The estimates obtained above for $R_{k}$ hold also for $\hat{R}_{k}$ with the obvious changes, and hence conditions \eqref{resample}, \eqref{second} are easily verified for $(R_{k},\hat{R}_{k})$. Consequently Theorem \ref{thm Hilbert
Polya} assures the existence of  $M=\lim M_{k}$ and $\hat{M}=\lim \hat{M}_{k}$ and that the conditional distribution of $X_{k+1}$ given $\mathcal F_k$ converges to $\mbox{Normal}[0,M,\hat{M}]$.

\subsection{Example: Hessenberg matrices with independent entries}

Combining the previous arguments with the central limit theorem, we can show that randomly scaled characteristic functions of the following  Hessenberg matrices converge almost surely to a random analytic function which is a mixture of Gaussian analytic functions.  
$$
\left(%
\begin{array}{ccccc}
  X     &   b_1&  &  & 0 \\
  X      & X &  b_2 &  &  \\
  \vdots &   &  & \ddots &  \\
         &   &  &  & b_{n-1} \\
  X     &   & \ldots &  & X \\
\end{array}%
\right)
$$
where $X$ are independent, identically distributed mean zero and finite fourth moment, and
$b_k =\sqrt{k} +O(k^{1/2-\eps})$. This condition could be
significantly weakened (as long as the rate of growth of
$b_{k}$'s is not too slow and not too fast); but this is not 
the main topic in the paper.

We have universality in the sense that the limit is a randomized 
Gaussian analytic function, even when the $X$ are not
Gaussian. In particular, the local behavior of the zeros is
universal; we expect that the probability of two zeros in
an disk of radius $\eps$ decays like $\eps^{6}$.

The Central Limit Theorem holds because the actual conditional covariance
matrix has to be scaled by $k$ to converge -- in
particular, a given characteristic function value has to
use more and more of the independent $X$-es. Therefore it
must be asymptotically normal; we omit the details.

\section{P\'olya's urn in Hilbert space -- proof} 

The goal of this section is to prove Theorem \ref{thm Hilbert Polya}. 

We recall a few facts about Hilbert-space valued 
random variables and martingales. 
\begin{fact} \label{hfacts}
In the following, $X_n$ will be a Hilbert-space valued
random sequence.

\renewcommand{\theenumi}{(\roman{enumi})}
\begin{enumerate}
\item \label{hmart} If $X_n$ is a martingale, and
 $\sum_{n=1}^{\infty} \ev \|X_{n+1}-X_{n}\|^2$ is finite, then there exists a limit $X$
 so that
 $\|X_n-X\|\to 0$ almost surely and in $L^2(\Omega)$.

\item \label{hnssum} If \ $\sum_n \ev \|X_n-X\|^2$ is finite, then
$\|X_n-X\|\to 0$ almost surely and in $L^2(\Omega)$.

\item \label{hsum} If \ $\sum \left(\ev \|X_n\|^2 \right)^{1/2}$ is finite,  then there exists a limit $S$ so that $\|S-\sum_{k=1}^n
X_k\|\to 0$ almost surely and in $L^2(\Omega)$.
\end{enumerate}
\end{fact}
\begin{proof}

 \ref{hnssum} The claim implies that $\sum_n \|X_n\|^2$ is
 finite almost surely (since its expected value is), so $\|X_n\|\to
 0$ almost surely.

 \ref{hsum} By the triangle inequality, we see that $X_n$
is a Cauchy sequence in $L^2(\Omega)$ and therefore it
converges to some $S$. Further,
\[
 \|S-\sum_{k=1}^n X_k\| \le
 \sum_{k=n+1}^{\infty} \|X_k\|
\]
 the latter being a monotone sequence that converges to
$0$ in expectation, hence almost surely.
\end{proof}

We are now ready for the main proof. 

\begin{proof}[Proof of Theorem \ref{thm Hilbert Polya}] Let
$$M_{k+1}-M_k = A_k+B_k, \qquad A_k= M_{k+1}-\ev[M_{k+1}|\mathcal
F_k], \qquad B_k=\ev[M_{k+1}|\mathcal F_k] -M_k
$$
be the Doob decomposition of the process $M_k$; the $A_k$
are martingale increments and the $B_k$ are increments of a
predictable process. It suffices to show that $\sum A_k$
and $\sum B_k$ converge. We have
$$
A_k = \frac{R_{k+1}-\ev [R_{k+1}\,|\,\mathcal F_k]}{k+1}
$$
which gives
\begin{eqnarray}\nonumber
\ev[\|A_k\|^2|\mathcal F_k] &=&
\frac{1}{(k+1)^2}\left(\ev[\|R_{k+1}\|^2\,|\,\mathcal F_k]-
\|\ev[R_{k+1}|\mathcal F_k]\|^2\right) \\
&\le& \frac{1}{(k+1)^2}\ev[\|R_{k+1}\|^2\,|\,\mathcal F_k]
=\frac{1}{k^2}O(1+ \|M_{k-1}\|^2+ \|M_k\|^2) \label{in1}
\end{eqnarray}
and
\begin{eqnarray}\label{in2}
B_k = \frac{\ev [R_{k+1}-M_k|\mathcal F_k]}{k+1}, \qquad
\|B_k\|\le (1+\|M_k\|) O(\eps_k/k)+ \|R_k\|O(1/k^2)
\end{eqnarray}
First we will show that $\ev \|M_k\|^2$ is bounded. Then it
will follow from Fact~\ref{hfacts}~\ref{hmart},  that $\sum A_k$ converges in $L^2$ and a.s.,
because it is an $L^2$-bounded martingale. Then we will
show that $\sum (\ev \|B_k\|^2)^{1/2}$ is finite. Fact~\ref{hfacts}~\ref{hsum} then 
implies that $B_1+\ldots+B_k$ is Cauchy in $L^2$ and $\ev
\sum \|B_k\|$  is finite. Thus  $\sum B_k$ converges a.s.
and in $L^2$. We write
\begin{equation}\label{Mk+1}
\ev \|M_{k+1}\|^2 =\ev \|M_k\|^2 + \ev \|A_k\|^2 + \ev
\|B_k\|^2 + 2 \Re\left\{\ev \langle M_k,B_k\rangle \right\} + 2 \ev \langle
M_k+B_k,A_k\rangle
\end{equation}
the last term vanishes, because $\ev[A_k|\mathcal F_k]=0$
and $M_k,B_k\in \mathcal F_k$. By \eqref{in2} we have
\begin{eqnarray}\nonumber
\ev \|B_k\|^2 &\le& (1+ \ev \|M_k\|^2)O(\eps_k^2/k^2) + \ev
\|R_k\|^2 O(1/k^4) \\ &\le& (1+ \ev \|M_{k-1}\|^2+\ev
\|M_k\|^2) O(\eps_k^2/k^2+1/k^4) \label{norm Bk squared}
\end{eqnarray}
The last inequality follows from the bound \eqref{second}.
Formula \eqref{norm Bk squared} and Cauchy-Schwarz imply
that if $\ev \|M_k\|^2$ is bounded then $\sum \ev \|B_k\|$
is finite.

 Further,  two applications of Cauchy-Schwarz give
\begin{eqnarray} \nonumber
|\ev \langle M_k,B_k \rangle| &\le& \ev \|M_k\|\|B_k\|
\\&\le& \left(\ev \|M_k\|^2 \ev \|B_k\|^2  \right)^{1/2}
\nonumber
\\&\le& (1+\ev
\|M_k\|^2 + \ev \|M_{k-1}\|^2)O(\eps_k/k+1/k^2)
\label{MkBk}
\end{eqnarray}
where the last inequality follows from \eqref{norm Bk
squared}. Using  \eqref{in1}, \eqref{norm Bk squared} and
\eqref{MkBk} in \eqref{Mk+1} we finally get
$$
\ev \|M_{k+1}\|^2 \le\ev \|M_k\|^2 + (1+\ev \|M_k\|^2 + \ev
\|M_{k-1}\|^2)O(\eps_k/k+1/k^2),
$$
so with the notation $$y_k=\max_{1\le \ell \le k} \ev
\|M_{\ell}\|^2$$ we have
\begin{eqnarray*}
y_{k+1} &\le& y_k + (1+y_k +y_{k-1})O(\eps_k/k+1/k^2) \\
&\le& y_k(1+O(\eps_k/k+1/k^2) )
\end{eqnarray*}
which shows
$$
y_k \le y_1 \prod_{\ell=2}^k (1+O(\eps_\ell/\ell+1/\ell^2))
$$
in particular $\ev \|M_k\|^2$ is bounded, completing the
proof.
\end{proof}

\section{Analysis of the Ginibre ensembles}
In this section we  prove Theorem~\ref{thm:convergenceofQn} and Theorem~\ref{thm:convergenceofQnforrealginibre}. By the Lanczos algorithm as explained in Section~\ref{sec:recursionofcharpoly},  it suffices to show that the characteristic polynomials of the Hessenberg matrices on the left side of \eqref{eq:hessenbergformofginibre} converge in distribution to a random analytic function $Q$, and that there exists a random positive definite Hermitian
function $K(z, w)$ so that given $K$ the function $Q(z)$ is a
Gaussian field with covariance kernel $K$. 

  If we condition on all  $b_k^2:=\frac{1}{\beta}\chi_{\beta k}^2$ variables on the first super diagonal, then the matrix  reduces to the semi-random Hessenberg matrix of Example~\ref{ex:semirandom}. We leave it to the reader to check that $k^{-2}|b_k^2-k|$ is summable, a.s., so that the conclusions of  Example~\ref{ex:semirandom} apply. Thus, conditional on all $b_k^2$, the limit $Q$ of $Q_n$ (in distribution) exists, and is a mixture of Gaussian analytic functions. This implies that unconditionally also, $Q$ is a mixture of Gaussian analytic functions.

   But we want to be able to study the properties of the random covariance kernel. Hence we directly analyze the Hessenberg matrix with the chi-squared variables in place, and get the same conclusions.

\begin{proof}[Proof of Theorem \ref{thm:convergenceofQn}]

Consider the nested matrices
$$
\left(%
\begin{array}{ccccc}
  \g     &   b_1&  &  & 0 \\
  \g      & \g  &  b_2&  &  \\
  \vdots &   &  & \ddots &  \\
         &   &  &  & b_{n-1} \\
  \g     &   & \ldots &  & \g \\
\end{array}%
\right)
$$
where the $\mathcal N$ refer to different  real $(\beta=1)$ or complex $(\beta=2)$ normal random
variables and $b_k^2$ to  Gamma$(k\beta/2,\beta/2)$ random variables, and all  entries are independent. Note that Gamma$(k\beta/2,\beta/2)$ is the distribution of the length-squared of a
$k$-dimensional standard real or complex Gaussian vector, for $\beta=1,2$
respectively.

Define $\overrightarrow{\g_k}$ as the vector formed by the first
$k$ entries of the $k$th row of the nested matrices. As in Example~\ref{ex:semirandom},  for $\lambda\in \mathbb C$, let
$X_1(\lambda)=1$, and for $k<n$, recursively define $X_{k+1}(\lambda)$  as
the solution to
$$
\overrightarrow{\g_k}\cdot \overrightarrow{X_k}(\lambda) +
b_kX_{k+1}(\lambda) = \lambda X_k(\lambda)
$$
where
$$
\overrightarrow{X_k}(\lambda)= (X_1(\lambda),\ldots,
X_k(\lambda)).
$$
Then $X_{k+1}$ is the characteristic function of the top $k\times k$ principal submatrix.

We note that
$$
\ev b_k^{-2} = \left\{\begin{array}{ll}
                             \frac{1}{k-1}, & \beta =2 ,\; k\ge 2 \\
                             \frac{1}{k-2}, & \beta =1 ,\;  k\ge 3
                           \end{array}\right.
$$
and
$$
\ev b_k^{-4} = \left\{\begin{array}{ll}
                             \frac{1}{(k-1)(k-2)}, & \beta =2 ,\; k\ge 3 \\
                             \frac{1}{(k-2)(k-4)}, & \beta =1 ,\;  k\ge
                             5
                           \end{array}\right.
$$
And in light of this we will set $k_0=1+4/\beta$. Let
$\mathcal F_k$ be the sigma-field generated by the first
$k$ rows of the matrix. Define
$$R_k(\lambda,\mu)= (X_k\otimes X_k^*)(\lambda,\mu):=X_k(\lambda)\overline{X_k(\mu)}, \;\; \hat{R}_k(\lambda,\mu)= (X_k\otimes X_k)(\lambda,\mu):=X_k(\lambda)X_k(\mu). $$

Then,  given $\mathcal F_k$ and $b_k$ we have
 \begin{eqnarray}
 X_{k+1} &\sim& \mbox{Normal} \left[ \frac{\lambda}{b_k}X_k, \frac{k}{b_k^2}M_k, \frac{k}{b_k^2}\hat{M}_k \right], \qquad \mbox{ where }  \label{e:meancov} \\
 M_k=\frac{R_1+\ldots+ R_k}k, && \hat{M}_k=\begin{cases}(\hat{R}_1+\ldots+ \hat{R}_k)/k & \mbox{ if }\beta=1,\\ 0 & \mbox{ if }\beta=2. \end{cases} \nonumber
 \end{eqnarray}
This means that conditionally on $\mathcal F_k$ and $b_k$
the random variable $X_{k+1}$ is a Gaussian field with
 the given mean and covariance.

In order to set  up a Hilbert space, we first fix a closed
disk $D\subset \mathbb C$. Recall the Hilbert space $\mathcal H$ defined in \eqref{e:kernelnorm} and the spaces $\mathcal H_{2}=\mathcal H$ and $\mathcal H_{1}=\mathcal H\times \mathcal H$ that were introduced in section~\ref{ex:semirandom}. Regard $R_k,M_k,\hat{R}_{k},\hat{M}_{k}$ as 
$\mathcal H$-valued random variables, and $X_k$ as an  $L^2(D)$-valued random variable.

Theorem~\ref{thm Hilbert Polya} will be applied to the $\mathcal H_{2}$-valued random variables $R_{k}$ for $\beta=2$ and to the $\mathcal H_{1}$-valued random variables $(R_{k},\hat{R}_{k})$ for $\beta=1$.
 Then \eqref{resample} holds as
\begin{equation}\label{eq:conditionalmeanofRk}
 \ev [R_{k+1} \, | \, \mathcal{F}_k] =\ev[b_k^{-2}]
\left( kM_k + \lambda\overline \mu X_kX_k^* \right)=
\frac{k}{k-\frac{2}{\beta}}M_k + \frac{\lambda\overline
\mu}{k-\frac{2}{\beta}} R_k
\end{equation}
where $\lambda, \mu$ are the arguments of $R_{k+1}$.
Thus
 \begin{eqnarray*} \|\ev [R_{k+1} \, | \, \mathcal{F}_k]-M_k\|
 &=&
\|M_k\|O(1/k)  + \|R_k\|O(1/k)
 \end{eqnarray*}
 which proves condition~\eqref{resample} for $\beta=2$. 
 Similarly, $\|\ev [\hat{R}_{k+1} \, | \, \mathcal{F}_k]-\hat{M}_k\|
 =
\|\hat{M}_k\|O(1/k)  + \|\hat{R}_k\|O(1/k)$ from which we get
\begin{eqnarray*} \|\ev [(R_{k+1},\hat{R}_{k+1}) \, | \, \mathcal{F}_k]-(M_k,\hat{M}_{k})\|
 &=&
\|(M_k,\hat{M}_{k})\|O(1/k)  + \|(R_k,\hat{R}_{k})\|O(1/k)
 \end{eqnarray*}
 which proves condition~\eqref{resample} for $\beta=1$.
 
We now proceed to check condition \eqref{second}. First, by Lemma \ref{l:knorm} we have 
\begin{eqnarray*}
\ev [ \|R_{k+1}\|^2 \, | \, \mathcal F_k,b_k] &\le &
\frac{24k^2}{b_k^4}(1+|D|)\|M_k\|^2 + \frac{8}{b_k^{4}}(\|(\lambda X_{k})^2\|^2+\|(\lambda X_{k})^4\|)
\\
&\le &
\frac{c}{b_k^4} \left(k^2\|M_k\|^2 + \|X_{k}^2\|^2+\|X_{k}^4\|)\right)
\\
&=&
\frac{c}{b_k^4} \left(k^2\|M_k\|^2 + \|R_{k}\|^2\right),
\end{eqnarray*}
where $c$ depends on $D$ only, and the last equality follows from \eqref{e:rank 1 H-norm}. Writing $R_k$ as $kM_k-(k-1)M_{k-1}$ we get the upper bound 
\begin{eqnarray*}
\ev [ \|R_{k+1}\|^2 \, | \, \mathcal F_k,b_k] &\le &
\frac{c'k^2}{b_k^4} (\|M_k\|^2 + \|M_{k-1}\|^2).
\end{eqnarray*}
we complete checking condition \eqref{second} for $\beta=2$ by noting that
\begin{eqnarray*}
\ev [ \|R_{k+1}\|^2 \, | \, \mathcal F_k] &=& \ev[\ev [ \|X_{k+1}\|^4 \, | \, \mathcal F_k,b_k]|\mathcal F_k] \\
&\le& 
c'k^2\ev[b_k^{-4}] (\|M_k\|^2 + \|M_{k-1}\|^2)\\
&=& O(\|M_k\|^2 + \|M_{k-1}\|^2)
\end{eqnarray*}
For $\beta=1$, we make similar computation for $\hat{R}_{k+1}$ and combine it with the above to verify  condition~\eqref{second}.

Thus, Theorem \ref{thm
Hilbert Polya} implies that $(M_k,\hat{M}_{k})$ converges almost surely
to a limit $(M,\hat{M})$. Of course $\hat{M}=0$ for  $\beta=2$. Since local $L^2$ convergence for analytic
functions implies sup-norm convergence, the limit $M$ is
analytic in the first variable and anti-analytic in the second variable while $\hat{M}$ is analytic in both its arguments. Also
$$
\frac{R_k}{k} = \frac{k M_k -(k-1)M_{k-1}}{k} \to 0
$$
and so $X_k/b_k\to 0$. Thus, the conditional distribution of
$X_{k+1}$ given $\mathcal F_k$ converges to $\mbox{Normal}[0,M,\hat{M}]$. For $\beta=2$ this is a  Gaussian analytic function with random covariance kernel $M$ while for $\beta=1$ it is a random analytic function with jointly Gaussian real and imaginary parts.
\end{proof}

\section{The mean of the random covariance kernel}
We have shown that the limit of characteristic polynomials is a centered Gaussian analytic function with random covariance kernels $M,\hat{M}$. Now we try to understand the distribution of $M$ itself. We are able to calculate the first two moments. We show the calculation for $M$ only, which is sufficient for $\beta=2$.

\noindent{\bf The expectation of $M(\lambda,\mu)$ :} 
From equation~(\ref{eq:conditionalmeanofRk}) we have
$$ \ev [R_{k+1} \, | \, \mathcal{F}_k] =\ev[b_k^{-2}]
\left( R_1+\ldots +R_k + \lambda\overline \mu R_k \right)
$$
so we would like to set $r_k$ to be the $\ev R_k$ to get
the recursion
\begin{equation}\label{eq:rrecursion}
r_{k+1} =\frac{ r_1+\ldots +r_k + \lambda\overline \mu r_k
}{k-\frac{2}{\beta}}
\end{equation}
but this is invalid, because $\ev[R_k]=\infty$ simply because $\ev[b_k^{-2}]=\infty$ for $k\le 2/\beta$.
Therefore we set
\[
 r_k=\ev [b_1^2b_2^2R_k] \hspace{3mm}\mbox{ for }\beta=1,\qquad r_k=\ev [b_1^2 R_k] \hspace{3mm} \mbox{ for }\beta=2.
\]
These $r_k$ are finite and the  recursion (\ref{eq:rrecursion})  above holds for these quantities, for $k>2/\beta$.

Fix $z=\lambda \overline \mu$ and use
\eqref{eq:rrecursion} to write
\begin{eqnarray*}
r_1+\ldots +r_k &=& (k-2/\beta)r_{k+1}-zr_k\\
 r_1+\ldots +r_{k-1}& =& (k-1-2/\beta)r_{k}-zr_{k-1}
\end{eqnarray*}
Taking differences we get

\[ r_{k+1}-r_k = \frac{z}{k-\frac{2}{\beta}}(r_k-r_{k-1}) =
\begin{cases}
(r_3-r_2) \frac{z^{k-2}}{(k-2)!}& \text {when } \beta =1 \\
(r_2-r_1) \frac{z^{k-1}}{(k-1)!}& \text {when } \beta =2
\end{cases}
\]
Summing these  we get
\begin{equation}\label{eq:expressionorrk}  \underline{\beta=2:} \;\;\;  r_{k+1} - r_1 = (r_2-r_1)\sum_{j=0}^{k-1} \frac{z^{j}}{j!}. \qquad
  \underline{\beta=1:} \;\;\;  r_{k+1} - r_2 = (r_3-r_2)\sum_{j=0}^{k-2} \frac{z^{j}}{j!}.
 \end{equation}
By direct computation we have the initial values
\[
r_1=\begin{cases} 1 & \mbox{ if }\beta=2. \\ 2 & \mbox{ if }\beta=1.\end{cases} \qquad
r_2=\begin{cases} 1+\lambda \overline \mu & \mbox{ if }\beta=2. \\ 2+2 \lambda \overline \mu & \mbox{ if }\beta=1.\end{cases} \qquad
r_3 = 2+2 \lambda \overline \mu + (\lambda \overline \mu)^2 \hspace{2mm} \mbox{ for  }\beta=1.
\]
 Plugging these into (\ref{eq:expressionorrk}), we get
\[ r_{k+1} = \begin{cases} 1 + z\sum_{j=0}^{k-1}\frac{z^j}{j!} \rightarrow 1+ze^z   & \mbox{ for }\beta=2. \\
2+2z+ z^2\sum_{j=0}^{k-2}\frac{z^j}{j!} \rightarrow 2+2z+z^2e^z   & \mbox{ for }\beta=1.
\end{cases}
\]
as $k\rightarrow \infty$. Thus the limiting covariance kernel satisfies
\[ \ev \left[M(\lambda,\mu)\right] = \begin{cases} 1+\lambda \overline{\mu}e^{\lambda \overline{\mu}} & \mbox{ for }\beta=2.  \\
2+2\lambda \overline{\mu}+(\lambda \overline{\mu})^2e^{\lambda \overline{\mu}} & \mbox{ for }\beta=1.
\end{cases}
\]
Contrast this with the planar Gaussian analytic function $f(z):=\sum_n a_n \frac{z^n}{\sqrt{n!}}$ which has covariance kernel $e^{z\overline w}$, for both real and complex i.i.d Gaussian coefficients $a_n$ (lest this sound like a contradiction, one must again consider the direct second moment  $\ev [f(z)f(w)]$ which is again equal to $e^{z\overline w}$ for the real case but vanishes identically in the complex case!).
%

\section{Second moment of the covariance kernel}

In this section we consider the second moment of the covariance kernel.
Just like the first moment case, we have to multiply by random constants
for this moment to exist. 

Fix $\beta=2$ and let $\alpha_k=\ev[b_1^4b_2^4R_k^2]$,
$\beta_k=\ev[b_1^4b_2^4S_k]$, where
$S_k=\sum\limits_{j=1}^{k-1}R_jR_k$. We shall get recursions for
these quantities, by first evaluating conditional expectations given
${\mathcal F}_k$. To this end, set $z=|\lambda|^2$ and observe
that
\begin{eqnarray*}
\ev[R_{k+1}^2|{\mathcal F}_k] &=& \ev[b_k^{-4}]\left(z^2R_k^2+2k^2M_k^2+4zkM_kR_k \right) \\
 &=& \frac{z^2R_k^2+2k^2M_k^2+4zkM_kR_k}{(k-1)(k-2)} \\
 &=& \frac{z^2R_k^2+2\sum_{j=1}^kR_j^2 + 4zS_k + 4zR_k^2+ 4\sum_{j=1}^kS_j}{(k-1)(k-2)}. \\
\ev[S_{k+1} |{\mathcal F}_k] &=& \sum_{i=1}^k R_i \ev[R_{k+1}|{\mathcal F}_k] \\
 &=& \left(\sum_{j=1}^k R_j\right)\frac{1}{k-1}\left(kM_k+zR_k\right) \\
 &=& \frac{1}{k-1}\left[ \left(\sum_{j=1}^kR_j\right)^2 + z S_k +z R_k^2\right] \\
 &=& \frac{\sum_{j=1}^k R_j^2 + 2\sum_{j=1}^k S_j + zS_k+z R_k^2}{k-1}.
\end{eqnarray*}
Multiplying by $b_1^4b_2^4$ and taking expectations, we get that
for $k\ge 3$
\begin{eqnarray*}
\alpha_{k+1} &=&
\frac{z(z+4)\alpha_k+2\sum_{j=1}^k\alpha_j+4z\beta_k+4\sum_{j=1}^k\beta_j}{(k-1)(k-2)}.
\\
\beta_{k+1} &=& \frac{\sum_{j=1}^k \alpha_j + 2\sum_{j=1}^k \beta_j + z\beta_k + z\alpha_k}{k-1}.
\end{eqnarray*}
These can be rephrased as
\begin{eqnarray*}
2\sum_{j=1}^k\alpha_j + 4\sum_{j=1}^k\beta_j &=& (k-1)(k-2)\alpha_{k+1}- z(z+4)\alpha_k -4z\beta_k. \\
\sum_{j=1}^k \alpha_j + 2\sum_{j=1}^k \beta_j  &=& (k-1)\beta_{k+1}- z\beta_k-z\alpha_k.
\end{eqnarray*}
For smaller $k$, it is possible to compute manually the following
values
$$
\begin{array}{c|c|c}
   k & \alpha_k & \beta_k \\ \hline
    1&12 & 0\\
    2 & 6\,\left( 2 + 4\,z + z^2 \right)  & 6\,\left( 1 + z \right) \\
    3 & 12 + 24\,z + 24\,z^2 + 8\,z^3 + z^4 & 2\,\left( 6 + 9\,z + 6\,z^2 + z^3 \right) \\
    4 & 72 + 144\,z + 144\,z^2 + 96\,z^3 + 33\,z^4 + 6\,z^5 + \frac{1}{2}z^6 &
   36 + 60\,z + 48\,z^2 + 24\,z^3 + \frac{11}2\,z^4 + \frac12z^5
\end{array}
$$
The $k=4$ cases in fact follow from the recursions above. For
$k\ge 4$, by writing the same equation for $k-1$ and taking the
difference we get the first two of the following equations. The
last one follows by taking the difference of the above two.
\begin{eqnarray*}
2\alpha_k+ 4\beta_k &=& (k-1)(k-2)\alpha_{k+1}-(k-2)(k-3)\alpha_k- z(z+4)\alpha_k+z(z+4)\alpha_{k-1} \\
 & &\hspace{1cm}-4z\beta_k+4z\beta_{k-1}. \\
\alpha_k + 2\beta_k &=& (k-1)\beta_{k+1}-(k-2)\beta_{k}- z\beta_k+z\beta_{k-1}-z(\alpha_k-\alpha_{k-1}). \\
2(k-1)\beta_{k+1}- 2z\beta_k &=& (k-1)(k-2)\alpha_{k+1}- z(z+4)\alpha_k -4z\beta_k+2z\alpha_k.
\end{eqnarray*}

Using symbolic computations, we found that the coefficients of
$z^j$ is the same in all $a_k$ for $k\ge j+3$. Similarly, the
coefficients of $b_k-b_{k-1}$ stabilize  and they are exactly half
the corresponding stable coefficients of $a_k$. The first few
coefficients are
$$a_\infty(z)= 72 + 192\,z
+ \frac{802\,z^2}{3} + \frac{776\,z^3}{3} + \frac{3799\,z^4}{20} +
\frac{9967\,z^5}{90} +
  \frac{666847\,z^6}{12600} + \frac{11161\,z^7}{525} + \frac{474659\,z^8}{64800} + \ldots
$$
however, we have not been able to find a closed-form formula for
the coefficients or $a_\infty(z)$. The final answer is
$$\lim_{k\to\infty} \ev M_k(\lambda,\lambda)^2=a_\infty(|\lambda |^2)/2.$$

\noindent{\bf Special case $z=0$:} The equations become
\begin{eqnarray*}
2\alpha_k+ 4\beta_k &=& (k-1)(k-2)\alpha_{k+1}-(k-2)(k-3)\alpha_k. \\
\alpha_k + 2\beta_k &=& (k-1)\beta_{k+1}-(k-2)\beta_{k}. \\
2\beta_{k+1} &=& (k-2)\alpha_{k+1}.
\end{eqnarray*}
Substituting the last one into the second, we get
\[ (k-1)(k-2)\alpha_{k+1}=\alpha_k(2+2(k-3)+(k-2)(k-3)) \]
 which implies $\alpha_{k+1}=\alpha_k$. Then $\beta_k=\frac{k-3}{2}\alpha_3$ for all $k\ge 3$. Thus
 \[ \ev M_k^2 =\frac{1}{k^2}\sum_{j=1}^k \alpha_j^2+2\beta_j \rightarrow \alpha_3. \]


\bigskip

\noindent {\bf Acknowledgements.} The second author was supported by the NSERC Discovery Accelerator Grant and the Canada Research Chair program.

\bibliography{gingaf}

\bigskip

\noindent {\sc Manjunath Krishnapur.} Department of Mathematics, Indian Institute of Science, Bangalore 560012, India.
{\tt manju\@@math.iisc.ernet.in}

\bigskip

\noindent {\sc B\'alint Vir\'ag.} Departments of Mathematics
and Statistics. University of Toronto, M5S 2E4 Canada.
{\tt balint\@@math.toronto.edu}

\end{document}